\documentclass[11pt]{amsart}

\usepackage{amsmath}
\usepackage{amsfonts}
\usepackage{amssymb}
\usepackage{array}
\usepackage{multirow}
\usepackage{mathtools}
\usepackage[english]{babel}
\usepackage{hyperref}

\newtheorem{theorem}{Theorem}[section]
\newtheorem{lemma}[theorem]{Lemma}

\theoremstyle{definition}

\theoremstyle{remark}
\newtheorem{remark}[theorem]{Remark}

\newcommand{\Z}{{\mathbb{Z}}}
\newcommand{\Q}{{\mathbb{Q}}}
\newcommand{\R}{{\mathbb{R}}}
\newcommand{\Bv}{{\mathbf{v}}}
\newcommand{\OK}{{\mathcal K}}
\newcommand{\OO}{{\mathcal O}}

\newcommand{\floor}[1]{\left\lfloor#1\right\rfloor}
\newcommand{\prpr}{{\prime\prime}}
\newcommand{\p}{\mathfrak{p}}

\newcommand{\h}{\mathrm{h}}

\let\Lm=\Lambda

\let\abs=\envert
\let\bs=\backslash

\newcommand{\acr}{\newline\indent}

\begin{document}
\title{Ratios of two powers of van der Laan-Padovan numbers}
\author[Tomohiro Yamada]{Tomohiro Yamada*}
\address{\llap{*\,}Center for Japanese language and culture\acr
                   Osaka University\acr
                   562-8678\acr
                   3-5-10, Semba-Higashi, Minoo, Osaka\acr
                   JAPAN}
\email{tyamada1093@gmail.com}

\subjclass[2020]{11A05, 11B37, 11B83, 11D61, 11J86.}
\keywords{Exponential diophantine equation; Linear recurrence sequence; van der Laan-Padovan sequence}

\begin{abstract}
The van der Laan-Padovan sequence $P_n ~ (n=0, 1, \ldots)$ is defined by $P_0=1, P_1=P_2=0$, and
$P_{n+3}=P_{n+1}+P_n$ for $n=0, 1, \ldots$.
We determine all pairs $(P_m, P_n)$ satisfying
$P_m^b=2^{g_1} 3^{g_2} 5^{g_3} 7^{g_4}P_n^a$ for some integers $g_1, g_2, g_3, g_4$, $a$, and $b$.
More generally, for a linear recurrence sequence $u_n$ satisfying the dominant root condition
and a given set of primes $p_1, \ldots, p_k$,
there exist only finitely many pairs $(u_m, u_n)$ satisfying
$u_m^b=p_1^{g_1} \cdots p_k^{g_k}u_n^a$ for some integers $g_1, \ldots, g_k$, $a$, and $b$.
\end{abstract}

\maketitle

\section{Introduction}\label{intro}

An $r$-th order linear recurrence sequence $(u_n)$ is a sequence
determined by initial terms $u_0, \ldots, u_{r-1}$ and the recurrence relation
\begin{equation}\label{eq11}
u_{n+r}=s_{r-1}u_{n+r-1}+\cdots +s_0 u_n
\end{equation}
for $n\geq 0$ with some complex numbers $s_0, \ldots, s_{r-1}$ and its characteristic polynomial
$$P(X)=X^r-s_{r-1}X^{r-1}-\cdots -s_0=(X-\alpha_1)^{d_1} \cdots (X-\alpha_t)^{d_t},$$
where we write distinct roots of the polynomial $P(X)$ for $\alpha_1, \ldots, \alpha_t$.

As is well known (see for example, 1.1.6 of \cite{EPSW}), $u_n$ can be written in the form
\begin{equation}\label{eq12}
u_n=\sum_{i=1}^t q_i(n) \alpha_i^n, q_i(n)=\sum_{j=1}^{d_i} \kappa_{w_i+j} n^{j-1}.
\end{equation}
for certain real numbers $\kappa_1$, $\ldots$, $\kappa_r$, where $w_i=d_1+\cdots +d_{i-1}$.

Throughout this paper, we assume that $s_0$, $\ldots$, $s_{r-1}$, $u_0$, $\ldots$, $u_{r-1}$ are integers.

Arithmetic properties and diophantine equations concerning linear recurrence sequences
have been extensively studied.
Among such sequences, the most extensively studied ones are Lucas sequences
$u_n=(\alpha^n-\beta^n)/(\alpha-\beta)$ with $\alpha$ and $\beta$ distinct roots of a given quadratic equation $X^2-s_1 X-s_0=0$ and the discriminant $\Delta=s_1^2-4s_0$,
where $s_0$ and $s_1$ are relatively prime and $\alpha/\beta$ is not a root of unity.
For such sequences, Bilu-Hanrot-Voutier theorem \cite{BHV} states that, for $n\geq 30$, $u_n$
always has a prime factor which does not divide $u_m$ with $0\leq m<n$.
It is well known that a prime $p$ not dividing $s_0 s_1 \Delta$ always divides $u_{p-(\Delta/p)}$
for a given Lucas sequence $(u_n)$ (such properties of Lucas sequences are concisely explained in
\cite[Section 2.IV]{Rib}).
Hence, for any given Lucas sequence $(u_n)$ and prime numbers $p_1, \ldots, p_k$,
we can easily determine all integers in this sequence $(u_n)$ all of whose prime factors belong to $p_1, \ldots, p_k$.
Moreover, Bugeaud, Mignotte, and Siksek \cite{BMS} gives a practical way to determine all powers in a given Lucas sequences,
although they apply an elaborate combination of modular techniques via Frey curves and Baker's method.

Less is known on general linear recurrence sequences.
However, many results are known under some restrictions.
For an overview of arithmetic and other properties of linear recurrence sequences and their applications, we refer to \cite{EPSW}.
Earlier results are surveyed by Stewart \cite{Ste2}.

Assume that $\alpha_i/\alpha_j$ are not roots of unity for $1\leq i<j\leq r$.
Mahler proved that $\lim_{n\rightarrow\infty}\abs{u_n}=\infty$.
Later, van der Poorten and Schlickewei \cite{vdPS} proved that for any $\epsilon>0$, the inequality
$$\abs{u_n}>\abs{\alpha_1}^{n(1-\epsilon)}$$
holds for sufficiently large $n$ (a more accessible proof is given by Fuchs and Heintze \cite{FuH}).
They also proved that, writing $P(a/b)$ with $\gcd(a, b)=1$ for the largest prime factor of an integer $ab$,
$\lim_{n\rightarrow\infty}P(u_n)=\infty$.
Evertse \cite {Evtse1} proved that 
$\lim_{n\rightarrow\infty, n>s, u_s\neq 0}P(u_n/u_s)=\infty$.

We note that these results are ineffective.
Indeed, no effective version of even an inequality $\lim_{n\rightarrow\infty}\abs{u_n}=\infty$ has been known
applicable for every non-degenerate sequence $(u_n)$.

For special cases, effective results are known.
Mignotte \cite{Mig} proved that 
If $\abs{\alpha_1}=\cdots =\abs{\alpha_\ell}>\abs{\alpha_{\ell+1}}$ with $\ell\leq 3$, then
$$\abs{u_n}>c^\prime \abs{\alpha_1}^n/n^c$$
for $n\geq n_0$
whenever $q_1(n)\alpha_1^n+\cdots +q_\ell(n)\alpha_\ell^n\neq 0$,
where $c$, $c^\prime$, and $n_0$ are positive constants effectively computable in terms of
$\alpha_1, \ldots, \alpha_\ell$, $q_1, \ldots, q_\ell$.
Shparlinskii \cite[Theorem 1]{Shp} proved that
if $\abs{\alpha_1}>\abs{\alpha_2}$, then
$P(u(n))>C\log n$ for $n\geq n_0$,
where $C>0$ and $n_0$ are effectively computable in terms of the sequence $(u_n)$.

Stewart \cite[Theorem 4]{Ste1} proved that,
for an algebraic field $\OK$ of degree $d$ over $\Q$, a real algebraic number $\alpha$ in $\OK$
with $\abs{\alpha}>1$,
if an integer $u(n)$ can be written in the form
$$u(n)=f(n)\alpha^n+h(n),$$
where $f(n)$ is a nonzero polynomial with coefficients from $\OK$
and $\abs{h(n)}<\abs{\alpha}^{\gamma n}$ for some $\gamma<1$,
then
$$P(u(n))>(1-\epsilon)\log n$$
for $n>n_0$, where $n_0$ is an effectively computable constant depending on
$d, \alpha, f, \gamma$, and $\epsilon$.
Stewart \cite{Ste3} replaced this lower bound by $C\log n\log\log n/\log\log\log n$
with $n_0$ and $C$ effectively computable depending on
$\alpha, f$, and $\gamma$.

Peth\"{o} \cite{Pet} proved that, for a given second-order linear recurrence sequence $(u_n)$
with integral parameters satisfying quite natural conditions and a given finite set $S$ of primes,
$u_n=wx^q$ has only finitely many solutions in integers $n, w, x, q$ with $\abs{x}>1$, $q\geq 2$,
and $w$ composed of primes in $S$, which can be bounded by an effectively computable constant.
Bugeaud and Kaneko \cite{BK} proved that if $\abs{\alpha_1}>\abs{\alpha_2}$ and $s_0\neq 0$,
then there are only finitely many perfect powers in $(u_n)$
and their number can be bounded by an effectively computable constant.

On the other hand, Evertse \cite{Evtse2} had proved that
if no quotient $\alpha_i/\alpha_j$ with $i\neq j$ is a root of unity,
then $P(u_n)\rightarrow \infty$ together with $n$ and
$P(u_m/u_n)\rightarrow \infty$
when $m>n$, $u_n\neq 0$, and $m$ tends to infinity.
However, Evertse uses Schlickewei's $p$-adic subspace theorem \cite{Sch},
which makes these results ineffective.

Odjoumani and Ziegler \cite{OZ1} proved that, if $\abs{\alpha_1}>\abs{\alpha_2}>\abs{\alpha_3}$ or
$\abs{\alpha_1}>\abs{\alpha_2}$ and $\kappa_1$ and $\alpha_1$ are multiplicatively independent,
then for any prime $p$ outside an effectively determinable finite set,
$(u_n)$ contains at most one prime power $\pm p^m$ with $m\neq 0$.
They \cite{OZ2} extended this result by proving that under some additional conditions,
$(u_n)$ contains at most two integers $\pm p^a q^b$ with $p^a q^b>1$ composed by two given primes $p$ and $q$
unless $p$ or $q$ belongs an effectively determinable finite set.

Results of Odjoumani and Ziegler imply that, for any given linear recurrence sequence $u_n$
satisfying a certain condition, there are only finitely many triples $(m, n, \ell)$ of nonnegative integers
satisfying
\begin{equation}
u_m^a u_n^b u_\ell^c=1
\end{equation}
for some integers $a, b, c$.
G\'{o}mez Ruiz and Luca \cite{GL} proved that for given $k$ different binary recurrence sequences
$u_n^{(1)}, \ldots, u_n^{(k)}$ satisfying certain conditions,
there are only finitely many $k$-tuples $u_{n_1}^{(1)}, \ldots, u_{n_k}^{(k)}$ such that
\begin{equation}
u_n^{(1) a_1} \cdots u_n^{(k) a_k}=1
\end{equation}
for some integers $a_1, \ldots, a_k$.
Independently of them, we determined all integer solutions of
\begin{equation}
\sigma(2^m)^a \sigma(3^n)^b \sigma(5^\ell)^c=1
\end{equation}
with $m, n, \ell>0$ in \cite{Ymd}.

In this paper, we prove the following finiteness result on multiplicative relations
in a given linear recurrence sequence.
\begin{theorem}\label{th1}
Let $u_n$ be an $r$-th order linear recurrence sequence defined by \eqref{eq11}.
Assume that roots $\alpha_1, \ldots, \alpha_t$ of its characteristic polynomials satisfy
$\abs{\alpha_1}>\abs{\alpha_2}\geq \cdots \geq \abs{\alpha_t}$, $\kappa_1\neq 0$,
$\kappa_i\neq 0$ for some $i\geq 2$, and $d_1=1$
(we note that $\alpha_1$ must be real and $q_1(n)=\kappa_1$ under this assumption).

If $n>m$ and $u_n^a=p_1^{g_1} \cdots p_k^{g_k} u_m^b$
for some integers $g_1, \ldots, g_k, a, b$ with $a>0$ and $b\geq 0$, then $m$ and $n$ can be bounded by
an effectively computable constant as given in Theorem \ref{th41} explicitly.
\end{theorem}

\begin{remark}
The condition $b\geq 0$ is essentially not restrictive.
Indeed, if $b<0$, then $u_n^a$ divides $p_1^{g_1} \cdots p_k^{g_k}$ and 
therefore $u_n=p_1^{h_1} \cdots p_k^{h_k}$ for some integers $h_1, \ldots, h_k$.
\end{remark}

Although our bound given in Theorem \ref{th41} is considerably large, we can determine all solutions of \eqref{eq11}
with the aid of the lattice reduction method, when given parameters are not extremely large.
We would like to do it for the van der Laan-Padovan sequence
$$1, 0, 0, 1, 0, 1, 1, 1, 2, 2, 3, 4, 5, 7, 9, 12, 16, 21, 28, 37, 49, 65, 86, 114, 151, 200, \ldots$$
defined by
$$P_0=1, P_1=P_2=0$$
and
\begin{equation}
P_{n+3}=P_{n+1}+P_n
\end{equation}
for $n\geq 0$.
In literature, this sequence has been called the Padovan sequence after Richard Padovan.
However, Padovan himself attributes this sequence to Dom Hans van der Laan (see for example,
\cite[Chapter 5]{Pad}).
Indeed, this sequence has been implicitly introduced in a work \cite[Chapter 8]{vdL} of van der Laan.
So that, this sequence should be called the van der Laan-Padovan sequence.

We find all pairs $(m, n)$ of nonnegative integers $m, n$ satisfying
\begin{equation}\label{eq13}
P_n^a=2^{g_1} 3^{g_2} 5^{g_3} 7^{g_4} P_m^b.
\end{equation}

\begin{theorem}\label{th2}
If \eqref{eq13} holds for some integers $g_1, g_2, g_3, g_4, a, b$ with $a>0$,
then we must have $m, n\in\{1, 2, 4\}$, $m, n\in\{0, 3, 5, \ldots, 18, 20, 25, 36\}$, or $m, n\in\{21, 27, 49\}$.
\end{theorem}

Our PARI-GP script is available at
\url{https://drive.google.com/file/d/1_IIBLRToi9jB3FeSmOR3Qx5FH1KFjYGD/view?usp=sharing}.

\section{Preliminaries}
In this section, we shall introduce some notations and lemmas.

Let $\OK$ and $\OO$ denote the number field $\Q(\alpha_1)$ and its ring of integers respectively.

Moreover, we define the absolute logarithmic height $\h(\alpha)$ of an algebraic number $\alpha$ in $\OK$.
For an algebraic number $\alpha$ in $\OK$ and a prime ideal $\p$ over $\OK$ such that
$\alpha=(\zeta_1/\zeta_2) \xi$ with $\xi\in \p^k$ and $\zeta_1, \zeta_2$ in $\OO\bs \p$,
we define the absolute value $\abs{\alpha}_\p$ by
\[
\abs{\alpha}_\p=N\p^{-k}
\]
as usual, where $N\p$ denotes the norm of $\p$, i.e., the rational prime lying over $\p$.
Now the absolute logarithmic height $\h(\alpha)$ is defined by
\[
\h(\alpha)=\frac{1}{2}\left(\log^+ \abs{\alpha}+\log^+ \abs{\bar\alpha}+\sum_\p \log^+\abs{\alpha}_\p\right),
\]
where $\log^+ t=\max\{0, \log t\}$ and $\p$ in the sum runs over all prime ideals over $\OK$.

In order to obtain an upper bound for the size of solutions, we use an lower bound for linear forms of logarithms due to Matveev \cite[Theorem 2.2]{Mat}.

\begin{lemma}\label{lm21}
Assume that $\OK$ is an real algebraic field of degree $D$.
Let $\alpha_1, \alpha_2, \ldots, \alpha_n$ be algebraic integers in $\OK$
which are multiplicatively independent
and $b_1, b_2, \ldots, b_n$ be arbitrary integers with $b_n\neq 0$.
Let \\ $A(\alpha)=\max\{D\h(\alpha), \abs{\log\alpha}\}$ and $A_j=A(\alpha_j)$.

Put
\begin{equation}
\begin{split}
B= & ~ \max \{1, \abs{b_1}A_1/A_n, \abs{b_2}A_2/A_n, \ldots, \abs{b_n} \},\\
\Omega= & ~ A_1A_2\ldots A_n,\\
C(n)= & ~ \frac{16}{n!}e^n(2n+3)(n+2)(4(n+1))^{n+1} \\
& \times \left(\frac{1}{2}en\right)(4.4n+5.5\log n+7+2\log D+\log(1+\log D)), \\
c_1 = & ~ 1.5eD(1+\log D)
\end{split}
\end{equation}
and
\begin{equation}
\Lm=b_1\log \alpha_1+\ldots+b_n\log \alpha_n.
\end{equation}
Then we have $\Lm=0$ or
\begin{equation}
\log\abs{\Lm}>-C(n)\Omega\log (c_1 B).
\end{equation}
\end{lemma}

\begin{remark}
When $\OK=\Q$, Aleksentsev's result in \cite{Ale} gives a better estimate.
However, the upper bound for our problem derived by Aleksentsev's result would be still considerably large.
\end{remark}

Let $$\Lm=x_1\theta_1+\cdots x_n\theta_n$$ be a linear form with $n\geq 2$, $x_1, \ldots, x_n\in\Z$, and $\theta_1, \ldots, \theta_n\in\R$.
Let $M=(m_{ij})$ be the $n$-th order square matrix defined by $m_{ij}=0$ for $1\leq i\leq n-1$, $1\leq j\leq n$ with $i\neq j$,
$m_{11}=\cdots =m_{n-1, n-1}=\gamma$, and $m_{ni}=\floor{C\gamma\theta_i}$ for $i=1, \ldots, n$,
where $C$ and $\gamma$ are constants chosen later,
and $\Bv_i$ be the $i$-th column vector $M$.
Let $l(M)$ be the the shortest length of vectors in the lattice generated by the column vectors
$\Bv_1$, $\ldots$, $\Bv_n$ of $M$.

We prove the following relation of a lower bound for $\abs{\Lm}$ and a lower bound for $l(M)$.
This is implicit in the proof of Lemma 3.7 of \cite{Weg2} and used in our previous work \cite{Ymd}.
However, we would like to make it explicit here and give its proof.
\begin{lemma}\label{lm22}
If $X_1$ be a real number such that $l(M)>X_1\sqrt{(n+1)^2+(n-1)\gamma^2}$ and $X_1\geq \max\{x_1, \ldots, x_n\}$, then $\abs{\Lm}>X_1/(C\gamma)$.
\end{lemma}
\begin{proof}
We proceed as in the proof of Lemma 3.7 of \cite{Weg2}.
We put
$$\Bv=\sum x_i \Bv_i=\begin{pmatrix}\gamma x_1 \\ \vdots \\ \gamma x_{n-1} \\ \tilde\Lm \end{pmatrix},$$
where $\tilde\Lm=x_1\floor{C\gamma\theta_1}+\cdots +x_n\floor{C\gamma\theta_n}$.
Then
\begin{equation}
\abs{\Bv}^2=\gamma^2 \sum_{i=1}^{n-1}x_i^2+\tilde\Lm^2
\end{equation}
and
\begin{equation}
\abs{\tilde\Lm-\gamma C\Lm}\leq \sum_{i=1}^n \abs{x_i} \abs{\floor{C\gamma\theta_i}-C\gamma\theta_i}< \sum_{i=1}^n \abs{x_i}\leq nX_1.
\end{equation}
By the assumption, we have
\begin{equation}
\abs{\gamma C \Lm}\geq \abs{\tilde\Lm}-\abs{\tilde\Lm-\gamma C\Lm}\geq \sqrt{l(M)^2-(n-1)\gamma^2 X_1^2}-nX_1\geq X_1,
\end{equation}
which proves the lemma.
\end{proof}

Finally, we prove the following fact.
\begin{lemma}\label{lm23}
$\kappa_1\in\Q(\alpha_1)$.
\end{lemma}
\begin{proof}
We may assume that $\alpha_1, \ldots, \alpha_{r_0}$ are the conjugates of $\alpha_1$ without the loss of generality.
Let $A$ be the extended vandermonde matrix of order $r$ with its $n$-th row consisting of $n^k \alpha_i^n$ for $i=1, \ldots, t$ and $k=0, \ldots, d_i-1$
and $A_{ij}$ be its $(i, j)$-cofactor.
Then, $\kappa_1$ can be written in a linear form $\kappa_1=B_1 s_0+\cdots +B_r s_{r-1}$ of $s_0, \ldots, s_{r-1}$,
where each $B_i=\det A_{i1}/\det A$.

We see that $\det A$ can be represented as a polynomial of $n^k \alpha_\ell^n$'s with $\ell=1, \ldots, t$, $k=0, \ldots, d_i-1$ alternating over $(k, l)'s$
and each $\det A_{i1}$ can be represented as a polynomial of $n^k \alpha_\ell^n$'s with $\ell=2, \ldots, t$ and $k=0, \ldots, d_i-1$ alternating over $(k, l)'s$
(explicit constructions for $\det A$ are given in \cite{Bar} and \cite{FlH}).

Hence, each $B_i$ can be written in a rational function of $\alpha_1$, $\ldots$, $\alpha_t$ with rational coefficients
which is symmetric on $\alpha_i$'s in other conjugate classes than the class of $\alpha_1$ and symmetric on $\alpha_2, \ldots, \alpha_{r_0}$.
In other words, each $B_i$ can be written in a rational function of $\alpha_2, \ldots, \alpha_{r_0}$ whose coefficients belongs to $\Q(\alpha_1)$.
This means that each $B_i$ itself belongs to $\Q(\alpha_1)$ and so does $\kappa_1$.
\end{proof}

\section{A preliminary upper bound}

Let $(u_n)$ be an $r$-th order linear recurrence sequence $u_n$ satisfying \eqref{eq11}
and assume that the characteristic roots $\alpha_1, \ldots, \alpha_t$ satisfy
$\abs{\alpha_1}>\abs{\alpha_2}\geq \cdots \abs{\alpha_t}$.
We write $u_n$ in the form \eqref{eq12}.

\begin{lemma}\label{lm31}
Let $m_1$, $\ldots$, $m_k$ be rational integers such that $m_1$, $\ldots$, $m_k$, $\alpha_1$, and $\kappa_1$ are multiplicatively independent and put
$A= \max\{\log m_1, \ldots, \log m_k, \\ \h(\alpha_1), \h(\kappa_1)\}$.
Let $n_0\geq 1$, $K>0$, $\delta>1$ be real numbers such that
\begin{equation}\label{eq30}
\abs{\frac{u_n}{\kappa_1\alpha_1^n}-1}<K\delta^{-n}
\end{equation}
for $n\geq n_0$ (we can take such real numbers $K$ and $\delta$ since $\abs{\alpha_1}>\abs{\alpha_2}$)
and $n_0>A/c_2$, where we put
$$c_2 =c_1 \max\{\h(\alpha_1), \h(\kappa_1)/n_0, \log \alpha_1+\max\{0, \log (K \kappa_1 (1+\epsilon))\}/n_0\}.$$ Then, we write $\epsilon=K\delta^{-n_0}$ and $\epsilon_1=\epsilon/(1-\epsilon)$.

Moreover, we put $D$ to be the degree of the field $\Q(\alpha_1)$ over $\Q$ and, for each $s=1, \ldots, k$,
\begin{equation*}
\begin{split}
A^{(s)}= & ~ \max\{\log m_1, \ldots, \log m_s, \h(\alpha_1), \h(\kappa_1)\}, \\
\Psi^{(s)} = &  ~(\log m_1) \cdots (\log m_s)\h(\alpha_1) \h(\kappa_1), \Psi^{\prime (s)}=\Psi / A^{(s)}, \\
c_3^{(s)} = & ~ \frac{\log((1+\epsilon_1)K)}{\Psi^{(s)} \log (c_2 n_0/A^{(s)})}, \\
C^{\prime (s)}= & ~ \frac{C(s+2) D^{s+2} +c_3^{(s)}}{\log\delta},
\end{split}
\end{equation*}
and
$C^{\prpr (s)}=C^{\prime (s)} \eta_1$, where we put $\eta_1$ to be the constant satisfying
$$\frac{\eta_1 Y\log Y}{\log(\eta_1 Y\log Y)}=Y$$
with $Y=c_2 C^{\prime (s)} \Psi^{\prime (s)}$.
We simply write $\Psi=\Psi^{(k)}$ and so on.

If $n\geq n_0$ and $u_n=m_1^{e_1} \cdots m_k^{e_k}$ for some integers $e_1, \ldots e_k$, then
\begin{equation}\label{eq31}
n<C^\prpr \Psi \log (c_2 C^\prime \Psi^\prime).
\end{equation}
\end{lemma}

\begin{remark}
If $P(X)$ has no double root, then we have $t=r$, $u_n=\sum_{i=1}^r \kappa_i \alpha_i^n$ for $n=0, 1, \ldots$,
and we can put $\delta=\abs{\alpha_1/\alpha_2}$ and
\begin{equation}
K=\frac{\abs{\kappa_2}+\cdots +\abs{\kappa_r}}{\abs{\kappa_1}}.
\end{equation}
\end{remark}

\begin{proof}
We put
$\Lm_0=e_1\log m_1+ \cdots +e_k\log m_k-\log\kappa_1-n\log\alpha_1=\log(u_n/\kappa_1\alpha_1^n)$.
It is clear that
\begin{equation}
\abs{\frac{u_n}{\kappa_1\alpha_1^n}-1}\leq \abs{\frac{\kappa_2}{\kappa_1}\left(\frac{\alpha_2}{\alpha_1}\right)^n+\cdots +\frac{\kappa_r}{\kappa_1}\left(\frac{\alpha_r}{\alpha_1}\right)^n}\leq K\delta^{-n}.
\end{equation}
We observe that
\begin{equation}
\abs{\log(1+x)}<\frac{\abs{x}}{1-\abs{x}}
\end{equation}
for $\abs{x}<1$ to obtain
\begin{equation}\label{eq32}
\abs{\Lm_0}<\frac{K\delta^{-n}}{1-K\delta^{-n}}\leq (1+\epsilon_1)K\delta^{-n}.
\end{equation}

Since $u_n\neq \kappa_1 \alpha_1^n$ by assumption, we have $\Lm_0\neq 0$.
We note that $D\h(\alpha_1)\geq \log\abs{\alpha_1}=\abs{\log\alpha_1}$
and $D\h(\kappa_1)\geq \log\abs{\kappa_1}=\abs{\log\kappa_1}$.
Moreover, $n\geq n_0>0$ by assumption.

Now we put $s$ to be the largest index $i$ such that $e_i>0$.
Since $\Q(\alpha_1)$ is a real field and $\kappa\in\Q(\alpha_1)$ by Lemma \ref{lm23},
we can apply Lemma \ref{lm21} for $\Lm_0$ with $\OK=\Q(\alpha_1)$, $n=s+2$, and $A_n=DA$ to obtain
\begin{equation}\label{eq33}
-\log\abs{\Lm_0}<C(s+2) D^{s+2} \Psi \log (c_1 B^{(s)}),
\end{equation}
where we observe that $\Omega=D^{s+2}\Psi$ and
\begin{equation}\label{eq34}
\begin{split}
B^{(s)}= ~ & \frac{\max\{e_1\log m_1, \ldots, e_k\log m_k, \h(\kappa_1), n\h(\alpha_1)\}}{A^{(s)}} \\
\leq ~ & \frac{\max\{n\h(\alpha_1), \h(\kappa_1), \log u_n\}}{A^{(s)}}
\leq \frac{c_2 n}{c_1 A^{(s)}}.
\end{split}
\end{equation}

From \eqref{eq32}, \eqref{eq33}, and \eqref{eq34}, we obtain
\begin{equation}
n\log\delta<\log (2K_1)+C(s+2) D^{s+2} \Psi^{(s)}\log \left(\frac{c_2 n}{A^{(s)}}\right),
\end{equation}
which immediately gives that
\begin{equation}
\frac{c_2 n}{A^{(s)}}<c_2 C^{\prime (s)} \Psi^{\prime (s)} \log \left(\frac{c_2 n}{A^{(s)}}\right).
\end{equation}
By the definition of $C^{\prpr (s)}$, we have
\begin{equation}
\frac{c_2 n}{A^{(s)}}<c_2 C^{\prpr (s)} \Psi^{\prime (s)} \log (c_2 C^{\prime (s)} \Psi^{\prime (s)})
\end{equation}
and therefore
\begin{equation}
n<C^{\prpr (s)} \Psi^{(s)} \log (c_2 C^{\prime (s)} \Psi^{\prime (s)}).
\end{equation}
We can easily see that the resulting constants $C^{\prpr (s)} \Psi^{(s)}$ and
$C^{\prime (s)}\Psi^{\prime (s)}$ grow as $s$ becomes large to obtain \eqref{eq31} as desired.
\end{proof}

\section{General upper bounds}

Let $p_1$, $\ldots$, $p_k$ be distinct primes and $x_0$ be the smallest prime other than them.
We take an arbitrary constant $\mu>1$.
We put $p_s$ to be the largest prime among $p_1$, $\ldots$, $p_k$ such that $g_s\neq 0$,
$A= \max\{\log p_s, \h(\alpha_1), \h(\kappa_1)\}$, and $C_0=\mu/(c_1\h(\alpha_1))$,
and $$\Psi=(\log p_1) \cdots (\log p_k)\h(\alpha_1)\h(\kappa_1).$$

\begin{theorem}\label{th41}
Let $n_1\geq 1$, $K>0$, $\delta>1$ be real numbers satisfying \eqref{eq30} for $n\geq n_1$
and $n_1>C_0 A$.
We write $\epsilon=K\delta^{-n_1}$ and assume that $\epsilon<1$.
Moreover, we write $\epsilon_1=\epsilon/(1-\epsilon)$
and $\epsilon_2=\max\{0, \log \kappa_1/(n_1\log\alpha_1)\}$.

We take $C_1^\prime=c_2 C^\prime$ with
$$C^\prime=\frac{C(k+2)D^{k+2}+\log((1+\epsilon_1)K)/(\Psi\log\mu)}{\log\delta}$$
and $C_1=C^\prpr$ obtained from $C^\prime$ as in Lemma \ref{lm31}.

Nextly, we put $C_2^\prime=c_2 C^\prime$ with
$$C^\prime=\frac{C(k+3)D^{k+3}+\log((1+\epsilon_1)K)/(\Psi(\log x_0)(\log\mu))}{\log\delta},$$
$C_2=C^\prpr$ obtained from $C^\prime$ as in Lemma \ref{lm31}, and $C_3=C_2 (1+\epsilon)$.
Moreover, We put
\begin{equation*}
\begin{split}
C_4= & ~ C(k+2)D^{k+2}, \\
C_5^\prime= & ~ c_1 \frac{C_3 \Psi^2 (\log\alpha_1) \log^2 (C_2^\prime \Psi)\max\{C_2 \h(\alpha_1), C_3\log\alpha_1\}}{A}, \\
C_5= & ~ C_4+
\frac{\log(2(1+\epsilon_1)KC_3 \Psi \log \alpha \log(C_2^\prime \Psi))}
{2C_4\Psi(\log C_5^\prime)(\log\delta)}, \\
C_6= & ~ 2C_5(1+\epsilon)(1+\epsilon_2)\log\alpha_1, \\
C_7= & ~ \eta_2 C_6,
\end{split}
\end{equation*}
where we put $\eta_2$ to be the constant satisfying
$$\frac{\eta_2 Y\log Y}{\log(\eta_2 Y\log Y)}=Y$$
with $Y=C_6\Psi$.

If $n>m\geq n_1$ and $u_n^a=p_1^{g_1} \cdots p_k^{g_k} u_m^b$
for some integers $g_1, \ldots, g_k, a, b$ with $a>0$ and $b\geq 0$, then we have either
\begin{equation}
m\leq n<2\max\{C_0, C_2 \Psi \log(C_2^\prime \Psi)\} C_5^\prime,
\end{equation}
or
\begin{equation}
m<2C_5\Psi\log(C_7\Psi\log(C_6\Psi))
\end{equation}
and
\begin{equation}
n<\max\{C_0, C_2 \Psi \log(C_2^\prime \Psi)\} C_7 \Psi \log(C_6\Psi).
\end{equation}
\end{theorem}

We begin by observing that $u_m, u_n>(1-\epsilon)\kappa_1\alpha_1^n>0$ from the assumption that
$n>m\geq n_1$ and $\epsilon<1$.
Moreover, we may assume that $\gcd(a, b)=1$.
Since $p_1, \ldots, p_k$ are distinct primes, we have
\begin{equation}\label{eq41}
u_m=p_1^{e_1} \cdots p_k^{e_k} x^a
\end{equation}
and
\begin{equation}\label{eq42}
u_n=p_1^{f_1} \cdots p_k^{f_k} x^b
\end{equation}
for some integer $x$ not divisible by $p_1$, $\ldots$, or $p_k$.
Indeed, we can write $u_m=p_1^{e_1} \cdots p_k^{e_k} U$ and
$u_n=p_1^{f_1} \cdots p_k^{f_k} V$ with $\gcd(U, p_1 \cdots p_k)=\gcd(V, p_1 \cdots p_k)=1$.
Then, we observe that $U^b=V^a$ and therefore $U=x^a$ and $V=x^b$ for some integer $x$.

If $x=1$ or $b=0$, then Lemma \ref{lm31} with $n_0=n_1$ gives that $m<n<C_1 \Psi \log (C_1^\prime \Psi)$,
where we note that $n_0\geq \mu A/c_2>A/c_2$ and $c_3\leq\log((1+\epsilon_1)K)/(\Psi\log\mu)$.

Hence, we may assume that $x\geq 2$ and $b>0$.
We put $A^\prime=\max\{A, \log x\}$.
Now, we would like to apply Lemma \ref{lm31} with $k+1$, $\Psi\log x$, and $A^\prime$
in place of $k$, $\Psi$, and $A$ respectively and with
\begin{equation*}
n_0=\max\{n_1, C_0 A^\prime\}, m_i= p_i ~ (i=1, \ldots, k), m_{k+1}=x.
\end{equation*}
Then, we see that $n_0\geq \mu A^\prime/c_2>A^\prime/c_2$ and, either $e_{k+1}=a>0$ or $e_{k+1}=b>0$ holds.
Moreover, we have $c_3\leq \log((1+\epsilon_1)K)/(\Psi(\log x_0)(\log\mu))$ in Lemma \ref{lm31}.
Hence, we can apply Lemma \ref{lm31} to obtain
\begin{equation}
\max\{m, n\}\leq \max\{n_0, C_2 \Psi(\log x)\log(C_2^\prime \Psi)\}.
\end{equation}
Since we have assumed that $n_1>C_0 A$, we have
\begin{equation}\label{eq43}
\max\{m, n\}\leq \max\{n_0, \max\{C_0, C_2 \Psi \log(C_2^\prime \Psi)\}\log x\}.
\end{equation}

Now, the theorem immediately follows when $\log x<C_5^\prime$.
Thus, we may assume that $\log x\geq C_5^\prime$.
We put
\begin{equation}
\begin{split}
\Lm= & ~ \log\left(\frac{u_m^b u_n^{-a}}{\kappa_1^{b-a}\alpha_1^{bm-an}}\right)
= \log\left(\frac{p_1^{g_1} \cdots p_k^{g_k}}{\kappa_1^{b-a}\alpha_1^{bm-an}}\right) \\
= & ~ \sum_{i=1}^k g_i\log p_i-(b-a)\log\kappa_1-(bm-an)\log\alpha_1.
\end{split}
\end{equation}
We can easily see that
\begin{equation}\label{eq44}
\abs{\Lm}\leq \frac{bK\delta^{-m}}{1-K\delta^{-m}}+\frac{aK\delta^{-n}}{1-K\delta^{-n}}
<(1+\epsilon_1)(a+b)K\delta^{-m}.
\end{equation}

If $\Lm=0$, then, from the assumption that $p_1, \ldots, p_k, \kappa_1, \alpha_1$ are
multiplicatively independent, we see that $b-a=bm-an=0$.
Since $m\neq n$, we must have $a=b=0$ while we have assumed that $a>0$.

Thus, we have $\Lm\neq 0$.
Hence, we can apply Lemma \ref{lm21} for $\Lm$ with $n=k+2$, $A_n=A$ noting that $g_s\neq 0$, and
\begin{equation}
B=\frac{\max\{g_1\log p_1, \ldots, g_k\log p_k, \abs{b-a}\h(\kappa_1), \abs{bm-an}\h(\alpha_1)\}}{A}
\end{equation}
to obtain
\begin{equation}
-\log\abs{\Lm}\leq C_4\Psi \log (c_1 B).
\end{equation}
Combining this bound with \eqref{eq44}, we have
\begin{equation}\label{eq45}
m\log \delta<\log((a+b)(1+\epsilon_1)K)+C_4\Psi \log (c_1 B).
\end{equation}

From \eqref{eq43}, we have
\begin{equation}\label{eq46}
\max\{a, b\}\leq C_3 \Psi (\log \alpha_1) \log(C_2^\prime \Psi)
\end{equation}
and the same upper bound also applies to $\abs{a-b}$ since $a$ and $b$ are both positive.
We can easily see that $g_i=be_i-af_i$ from \eqref{eq41} and \eqref{eq42} and
\begin{equation}
\abs{g_i}\leq C_3^2 \Psi_i \Psi (\log^2 \alpha_1) (\log x) \log^2 (C_2 \Psi)
\end{equation}
for $i=1, \ldots, k$, where $\Psi_i=\Psi/\log p_i$.
Moreover, combining \eqref{eq43} and \eqref{eq46}, we immediately see that
\begin{equation}
\abs{bm-an}\leq C_2 C_3 \Psi^2 (\log\alpha_1) (\log x) \log^2 (C_2 \Psi).
\end{equation}
Hence, we obtain
\begin{equation}
B\leq \frac{C_3 \Psi^2 (\log\alpha_1) (\log x) \log^2 (C_2^\prime \Psi)\max\{C_2 \h(\alpha_1), C_3\log\alpha_1\}}{A}
\end{equation}
and therefore $c_1 B\leq C_5^\prime \log x$.
From \eqref{eq46}, we see that $a+b<2C_3 \Psi (\log \alpha_1) \log(C_2^\prime \Psi)$.
Now \eqref{eq45} yields that
\begin{equation}\label{eq47}
m\log \delta<\log(2C_3(1+\epsilon_1)K\Psi (\log \alpha_1) \log(C_2^\prime \Psi))
+C_4\Psi \log (C_5^\prime \log x)
\end{equation}
and we conclude that
\begin{equation}\label{eq48}
m<C_5 \Psi\log(C_5^\prime \log x)<2C_5\Psi\log\log x.
\end{equation}

Now, observing that $u_m\geq x^a$ and $a\log x\leq\log u_m<(1+\epsilon)(\log\kappa_1+m\log\alpha_1)$,
we have
\begin{equation}\label{eq48a}
a\log x<C_6\Psi\log\log x<C_6\Psi\log(a\log x)
\end{equation}
and
\begin{equation}\label{eq49}
a\log x<C_7\Psi\log(C_6\Psi).
\end{equation}
Now, recalling the assumption that $a>0$, the theorem follows from \eqref{eq48} and \eqref{eq43} again.

\section{Proof of Theorem \ref{th2}}
We apply our method to the van der Laan-Padovan sequence.
We put $\alpha_1=1.324717\cdots $ be the plastic ratio, the real solution of the equation $X^3-X-1=0$,
$\alpha_2, \alpha_3=-0.662358\cdots \pm 0.562279\cdots i$ be the remaining solutions, which are mutually conjugate,
and $\kappa_i=1/(2\alpha_i+3)$ for $i=1, 2, 3$.
Now we can write $P_n=\kappa_1\alpha_1^n+\kappa_2\alpha_2^n+\kappa_3\alpha_3^n$.
We note that $K=5.599815\cdots$ and $\delta=1.524702\cdots$.

Now we assume that $m$ and $n$ are two nonnegative integers satisfying \eqref{eq13}
for some integers $g_1, g_2, g_3, g_4, a, b$ with $a\geq 0$.
Moreover, we may assume that $b>0$ without the loss of generality.

We begin by settling the special case $P_n=2^{f_1}3^{f_2}5^{f_3}7^{f_4}$.

\begin{lemma}\label{lm51}
If $P_n=2^{f_1}3^{f_2}5^{f_3}7^{f_4}$, then $n=0, \ldots, 18$, $20$, $25$, or $36$.
\end{lemma}
\begin{proof}
We proceed as in the proof of Theorem \ref{th41} taking $n_1=27$ and $\mu=10$.
Indeed, we have $A=\max\{\log 7, \h(\alpha_1), \h(\kappa_1)\}=\log 7$
and therefore $n_1>9>C_0 A$.
Thus we see that $\epsilon<6.3413\times 10^{-5}$ and $\epsilon_1<6.3417\times 10^{-5}$.

Now we apply Lemma \ref{lm31} with $A=\max\{\log 7, \h(\alpha_1), \h(\kappa_1)\}$ to obtain
$$n<C_1\Psi\log(C_1^\prime \Psi)<2.456\times 10^{24}$$
with $C_1<2.066058\times 10^{23}$ and $C_1^\prime<5.381845\times 10^{22}$.

We can easily see that $P_n=2^{f_1}3^{f_2}5^{f_3}7^{f_4}<(1+\epsilon)\kappa_1\alpha^n$
and therefore $f_i<2.456\times 10^{24}$ for each $i$.
Now, putting
$$\Lm_0=f_1\log 2+\cdots +f_4\log 7-\log\kappa_1-n\alpha_1,$$
we have $\abs{\Lm_0}<(1+\epsilon_1)K\delta^{-n}$ for $n\geq n_1$.

Taking $C=10^{150}$, $\gamma=11$, and $X_1=2.456\times 10^{24}$,
we have a reduced matrix $M$ such that $l(M)>X_1\sqrt{49+5\gamma^2}$.
Thus, Lemma \ref{lm22} gives that
$$\abs{\Lm}>\frac{X_1}{C\gamma}>2.233\times 10^{-127}$$
and $n\leq 695$.
Checking each $n$, we can prove the lemma.
\end{proof}

Now we proceed to the remaining case.
Then \eqref{eq13} holds with $b>0$ and we may assume that $\gcd(a, b)=1$.

\begin{lemma}\label{lm52}
If $P_n^a=2^{g_1} 3^{g_2} 5^{g_3} 7^{g_4} P_m^b$ with $\gcd(a, b)=1$, then $m\leq 988$.
Moreover, we must have $bm\leq 3.850562\times 10^{29}$ and $an<5.5553\times 10^{29}$.
\end{lemma}

\begin{proof}
We take $n_1=27$ and $\mu=10$ again.
Thus, we have $\epsilon<6.3413\times 10^{-5}$ and $\epsilon_1<6.3417\times 10^{-5}$
We begin by observing that $P_m=2^{e_1} \cdots 7^{e_4} x^a$ and
$P_n=2^{f_1}\cdots 7^{f_4} x^b$ for some integers $e_1, \ldots , e_4$,
$f_1, \ldots, f_4$, and $x\geq 11$.

We put
$$\Lm=g_1\log 2+\cdots +g_4\log 7+(a-b)\log\kappa_1+(an-bm)\kappa\alpha_1.$$
Instead of \eqref{eq49}, we obtain
\begin{equation}\label{eq51}
a\log x<C_5(1+\epsilon)(1+\epsilon_2)\Psi\log\alpha_1\log(C_5^\prime a\log x)
\end{equation}
with
$\epsilon_2=0$, $C_2<1.062372\times 10^{26}$, $C_2^\prime<7.587587\times 10^{24}$,
$C_3<1.062439\times 10^{26}$, $C_5<8.072767\times 10^{22}$, and $C_5^\prime<2.003782\times 10^{54}$.
Then, \eqref{eq51} implies that $a\log x<9.561\times 10^{23}$.

Hence, we must have $\log x\leq a\log x<9.561\times 10^{23}$.
From \eqref{eq43}, we obtain
$$an<\max\{an_0, C_2\Psi\log(C_2^\prime \Psi)(a\log x)\}<1.325038\times 10^{51}.$$
Moreover, from \eqref{eq48}, we obtain $m<3.399751\times 10^{24}$
and \eqref{eq46} gives $\max\{a, b\}<3.897329\times 10^{26}$.
Hence, we have $bm<1.325038\times 10^{51}$.

Now we can apply Lemma \ref{lm22} with $X_1=1.325038\times 10^{51}$.
Taking $C=10^{310}$ and $\gamma=2.6$,
we have a reduced matrix $M$ such that $l(M)>X_1\sqrt{49+5\gamma^2}$.
Hence, Lemma \ref{lm22} yields that $\abs{\Lm}>5.0963\times 10^{-260}$.

We observe that
$$\delta^m<1.962208(a+3.897329\times 10^{24})(1+\epsilon_1)K\times 10^{259}.$$
Since $a\log x<(1+\epsilon)m\log\alpha_1$,
we have $m\leq 1564$, $bm\leq 6.0955\times 10^{29}$, $a\log x<634.54$
and, using \eqref{eq43} again, $an<8.794\times 10^{29}$.

Now we use Lemma \ref{lm22} again but with $C=10^{183}$, $\gamma=4$, and $X_1=8.794\times 10^{29}$
to obtain $\abs{\Lm}>2.1985\times 10^{-154}$ and, proceeding like above, we have $m\leq 988$,
$bm\leq 3.850562\times 10^{29}$, $a\log x<400.85$, and $an<5.5553\times 10^{29}$.
This completes the proof of the lemma.
\end{proof}

Now we check each $m\leq 988$.
If $m\in \{0, \ldots, 18, 20, 25, 36\}$, then so must $n$.
Hence, we may assume that $m$ is equal to none of $0, \ldots ,18$, $20$, $25$, and $36$.
Now we play with the logarithmic form
$$\Lm_1=a\log\frac{P_n}{\kappa_1\alpha_1^n}=b\log P_m+g_1\log 2+\cdots g_4\log 7-an\log\alpha_1-a\log\kappa_1$$
instead of $\Lm$ for each $m\leq 988$ not equal to $0, \ldots, 18$, $20$, $25$, or $36$.
Indeed, for any such $m$,
we confirmed that Lemma \ref{lm22} works for some $C$ and $\gamma$ with $\gamma C\leq 3\times 10^{214}$
and therefore $\abs{\Lm_1}>1.85176\times 10^{-185}$.
Moreover, putting $H_m=P_m/(2^{e_1} \cdots 7^{e_4})$ with $H_m=0$ or $\gcd(H_m, 210)=1$,
we confirmed that $H_m$ can never be a perfect power for $0\leq m\leq 1012$ unless $H_m\in \{0, 1\}$.
Hence, we must have $a=1$.

Now we assume that $n>\max\{m, 100\}$.
Since
\begin{equation}
\abs{\Lm_1}=\abs{\log n-\log\kappa-n\log\alpha_1}<\frac{K(1+\epsilon_3)}{\delta^n}
\end{equation}
with $\epsilon_3<3\times 10^{-18}$, we have $n\leq 1012$.
We observe that $H_m=H_n$.
Indeed, if $H_m=1$, then $H_n=1$ noting that $H_m^b=H_n^a$.
Otherwise, $H_m$ can never be a perfect power for $0\leq m\leq 1012$ as mentioned above and
we must have $a=b=1$, which yields that $H_m=H_n$ again.
We confirmed that this can occur only when
$m, n\in\{1, 2, 4\}$ with $H_m=H_n=0$ ,
$m, n\in\{0, 3, 5, 6, \ldots, 18, 20, 25, 36\}$ with $H_m=H_n=1$,
or $m, n\in\{21, 27, 49\}$ with $H_m=H_n=13$.
This completes the proof of Theorem \ref{th2}.

{}
\end{document}